\documentclass{amsart}

\usepackage{amsthm} \usepackage{amssymb} \usepackage{color}
\usepackage{enumerate} \usepackage[all]{xy} \usepackage{verbatim}

\SelectTips{cm}{} \calclayout  

\numberwithin{equation}{section}

\theoremstyle{plain}

\newtheorem{lemma}{Lemma}[section]
\newtheorem{theorem}[lemma]{Theorem}
\newtheorem{proposition}[lemma]{Proposition}
\newtheorem{corollary}[lemma]{Corollary}

\theoremstyle{definition} 
\newtheorem{example}[lemma]{Example}

\theoremstyle{remark}

\newtheorem{remark}[lemma]{Remark} \newtheorem*{Claim1}{Claim 1}
\newtheorem*{Claim2}{Claim 2} \newtheorem*{Claim3}{Claim 3}
\newtheorem*{Claim4}{Claim 4} \newtheorem*{Claim5}{Claim 5}
 \newtheorem*{Baer}{Baer's criterion}

\def\urltilda{\kern -.15em\lower .7ex\hbox{\~{}}\kern .04em}

\newcommand{\cosupp}{\operatorname{cosupp}}

\newcommand{\Ext}{\operatorname{Ext}}

\newcommand{\hh}{\operatorname{H}} \newcommand{\zz}{\operatorname{Z}}
\newcommand{\bb}{\operatorname{B}}
\newcommand{\Hom}{\operatorname{Hom}}

\newcommand{\injdim}{\operatorname{inj\,dim}}
\newcommand{\projdim}{\operatorname{proj\,dim}}

 \renewcommand{\le}{\leqslant}
\renewcommand{\ge}{\geqslant}

\newcommand{\RHom}{\operatorname{{\mathbf R}Hom}}

\newcommand{\Spec}{\operatorname{Spec}}

\newcommand{\supp}{\operatorname{supp}}

\newcommand{\Tor}{\operatorname{Tor}}

\newcommand{\splf}{\operatorname{splf}}

\newcommand{\gs}{{\scriptscriptstyle{>}}}

\newcommand{\les}{{\scriptscriptstyle\leqslant}}
\newcommand{\lotimes}{\otimes^{\mathbf L}}

 \newcommand{\xra}{\xrightarrow}
\newcommand{\lra}{\longrightarrow}

  \def\mcU{\mathcal{U}}

\def\bbQ{\mathbb Q} \def\bbZ{\mathbb Z}

\newcommand{\fa}{\mathfrak{a}} \newcommand{\fb}{\mathfrak{b}}
 \newcommand{\fp}{\mathfrak{p}}

\title[Tests for injectivity of modules over commutative rings]{Tests
  for injectivity of modules\\ over commutative rings}

\author[L.\,W. Christensen]{Lars Winther Christensen}

\address{L.W.C. Texas Tech University, Lubbock, TX 79409, U.S.A.}

\email{lars.w.christensen@ttu.edu}

\urladdr{http://www.math.ttu.edu/\urltilda lchriste}

\author[S.\,B. Iyengar]{Srikanth B. Iyengar}
\address{S.B.I. University of Utah, Salt Lake City, UT 84112, U.S.A.}
\email{iyengar@math.utah.edu}
\urladdr{http://www.math.utah.edu/\urltilda iyengar}

\thanks{We thank the Centre de Recerca Matem\`atica, Barcelona, for
  hospitality during visits in Spring 2015, when part of the work
  reported in this article was done.  L.W.C.\ was partly supported by
  NSA grant H98230-14-0140, and S.B.I.\ was partly supported by NSF
  grant DMS-1503044.}

\date{15 June 2016}



\begin{document}

\begin{abstract}
  It is proved that a module $M$ over a commutative noetherian ring
  $R$ is injective if $\mathrm{Ext}_{R}^{i}((R/{\mathfrak
    p})_{\mathfrak p},M)=0$ holds for every $i\ge 1$ and every prime ideal
  $\mathfrak{p}$ in~$R$. This leads to the following characterization
  of injective modules: If $F$ is faithfully flat, then a module $M$
  such that $\Hom_R(F,M)$ is injective and $\Ext^i_R(F,M)=0$ for all
  $i\ge 1$ is injective. A limited version of this characterization is
  also proved for certain non-noetherian rings.
\end{abstract}

\maketitle

\section{Introduction}
Let $R$ be a commutative ring. In terms of cohomology, Baer's
criterion asserts that an $R$-module $M$ is injective if (and only if)
$\Ext_R^1(R/\fa,M)=0$ holds for every ideal $\fa$ in $R$. When $R$ is
also noetherian, it suffices to test against prime ideals and locally,
namely, $M$ is injective if either of the following conditions holds:
\begin{itemize}
\item $\Ext_{R}^{1}(R/\fp,M)=0$ for every prime ideal $\fp$ in $R$;
\item $\Ext_{R_\fp}^1(k(\fp),M_\fp)=0$ for every prime ideal $\fp$ in
  $R$.
\end{itemize}
Here, and henceforth, $k(\fp)$ denotes the field $(R/\fp)_\fp$. The
main result of this paper is that injectivity can be detected by
vanishing of Ext globally against these fields.

\begin{theorem}
  \label{thm:injdim}
  Let $R$ be a commutative noetherian ring and let $M$ be an
  $R$-complex. If for some integer $d$, one has
  \begin{equation*}
    \Ext^{i}_{R}(k(\fp),M)=0 \quad \text{for every prime ideal $\fp$ in
      $R$ and all $i>d$}\:,
  \end{equation*}
  then the injective dimension of $M$ is at most $d$.
\end{theorem}

As recalled in Example \ref{ex:QQ}, the module $\Ext_R^1(k(\fp),M)$
can be quite different from $\Ext_{R}^{1}(R/\fp,M)$ and
$\Ext_{R_\fp}^1(k(\fp),M_\fp)$. Nevertheless the appearance of
$\Ext_{R}(k(\fp),-)$ in this context is not unexpected in the light of
the recent work on cosupport of complexes in ~\cite{BIK-12}; see also
the discussion around Corollary~\ref{cor:detection}.

The proof of the theorem above is given in Section~\ref{se:proof}, and
applications are presented in Section~\ref{se:applications}. One such,
discussed in Remark~\ref{rmk:cobasechange}, is a characterization of
injectivity of an $R$-module $M$ in terms of that of $\Hom_R(F,M)$,
where $F$ is a faithfully flat $R$-module. In
Section~\ref{sec:nonnoetherian}, we establish a partial extension of
this last result to certain non-noetherian rings.

\section{Proof of Theorem~\ref{thm:injdim}}
\label{se:proof}
Our standard reference for basic definitions and constructions
involving complexes is \cite{LLAHBF91}. We first recall that as a
consequence of Baer's criterion, the injective dimension of an
$R$-complex is detected by vanishing of Ext against cyclic modules.

\begin{Baer}
  Let $R$ be a commutative noetherian ring, $M$ an $R$-complex, and
  $d$ an integer. One has $\injdim_{R}M\le d$ if and only if
  \begin{equation*}
    \Ext^{i}_{R}(R/\fa,M)=0\quad\text{for every ideal $\fa$ in $R$ and
      all $i>d$}\:.
  \end{equation*}
  This result is contained in \cite[Theorem 2.4.I]{LLAHBF91}.
\end{Baer}

Consider the collection of ideals
\begin{equation*}
  \mcU := \{\fa \subset R \mid \Ext^{i}_{R}(R/\fa,M)\ne 0\ \text{ for
    some $i> d$}\}\:.
\end{equation*}
If this collection is empty, then the desired inequality,
$\injdim_{R}M \le d$, holds by Bear's criterion. Thus, we assume that
$\mcU$ is non-empty and aim for a contradiction. It is achieved by
establishing a sequence of claims, the first of which is standard but
included for convenience.

\begin{Claim1}
  With respect to inclusion, $\mcU$ is a poset and its maximal
  elements are prime ideals.
\end{Claim1}

\begin{proof}
  Let $\fa$ be a maximal element in $\mcU$. Choose a prime ideal $\fp
  \supseteq \fa$ such that $\fp/\fa$ is an associated prime of
  $R/\fa$, and pick an element $r\in R$ be such that
  $\fp=(\fa\,\colon\,r)$. The ideal $\fa+(r)$ properly contains $\fa$
  and hence is not in $\mcU$. From the exact sequence of Ext modules
  associated to the standard exact sequence
  \begin{equation*}
    0 \lra R/\fp \lra R/\fa \lra R/(\fa+(r))\lra 0
  \end{equation*}
  it follows that $\fp$ is in $\mcU$. Since $\fa$ is maximal in
  $\mcU$, the equality $\fa=\fp$ holds.
\end{proof}

Fix a maximal element $\fp$ in $\mcU$; by Claim 1 it is a prime
ideal. Set $S:=R/\fp$ and let $Q$ be the field of fractions of the
domain $S$. We proceed to analyze the $S$-complex
\begin{equation*}
  X :=\RHom_{R}(S,M)\:.
\end{equation*}

\begin{Claim2}
  The natural map $\hh^{i}(X)\to Q\otimes_{S}\hh^{i}(X)$ is an
  isomorphism for all $i > d$.
\end{Claim2}

\begin{proof}
  Fix an element $s\ne 0$ in $S$. Let $x$ be an element in $R$ whose
  residue class mod $\fp$ is $s$. By the maximality of $\fp$, the
  ideal $\fp + (x)$ is not in $\mcU$. As one has $S/(s) \cong R/(\fp +
  (x))$, it follows that $\Ext_{R}^{i}(S/(s),M)=0$ holds for all
  $i>d$. Thus, applying $\RHom_{R}(-,M)$ to the exact sequence
  \[
  0 \lra S \xra{\ s \ } S \lra S/(s) \lra 0\:,
  \]
  shows that multiplication $\hh^{i}(X) \xra{\,s\,} \hh^{i}(X)$ is an
  isomorphism for $i> d$.
\end{proof}

In the derived category over $S$, consider the triangle defining
(soft) truncations
\begin{equation}
  \label{eq:truncate}
  \tau^{\les d} X \lra X \lra \tau^{\gs d} X\lra\:.
\end{equation}

\begin{Claim3}
  There is an isomorphism $\tau^{\gs d}X\cong \hh(\tau^{\gs d}X)$ in
  the derived category over $S$, and the action of $S$ on
  $\hh(\tau^{\gs d}X)$ factors through the embedding $S\to Q$.
\end{Claim3}

\begin{proof}
  It follows from Claim 2 that the canonical morphism $\tau^{>d}X \to
  Q \otimes_S \tau^{>d}X$ yields an isomorphism in the derived
  category over $S$. The right-hand complex is one of $Q$-vector
  spaces, so it is isomorphic to its homology, and another invocation
  of Claim 2 yields the claim.
\end{proof}

\begin{Claim4}
  One has $\injdim_{S}(\tau^{\les d}X)\le d$.
\end{Claim4}

\begin{proof}
  By Baer's criterion it suffices to show that
  $\Ext^{i}_{S}(S/\fb,\tau^{\les d}X)$ vanishes for every ideal $\fb$
  in $S$ and all $i>d$. Notice first that we may assume that $\fb$ is
  non-zero, because for $i>d$ one has
  \[
  \Ext^{i}_{S}(S, \tau^{\les d}X) \cong \hh^{i}(\tau^{\les d}X)=0\:,
  \]
  where the vanishing is by construction. For $\fb \ne 0$ one has $Q
  \otimes_S S/\fb=0$, and Claim 3 together with Hom-tensor adjunction
  yields
  \begin{align*}
    \Ext^{*}_{S}(S/\fb,\tau^{\gs d}X)
    & \cong \Ext^{*}_{S}(S/\fb,\hh(\tau^{\gs d}X)) \\
    & \cong \Ext^{*}_{Q}(Q\otimes_{S}S/\fb,\hh(\tau^{\gs d}X)) \\
    & = 0\,.
  \end{align*}
  For $i>d$ the exact sequence in homology associated to
  \eqref{eq:truncate} now gives the first isomorphism below
  \begin{align*}
    \Ext^{i}_{S}(S/\fb,\tau^{\les d}X)
    &\cong \Ext^{i}_{S}(S/\fb,X) \\
    &\cong \Ext^{i}_{R}(S/\fb,M) \\
    &\cong \Ext^{i}_{R}(R/\fa,M) \\
    & = 0\:.
  \end{align*}
  The second isomorphism follows from Hom-tensor adjunction and the
  definition of $X$. The next isomorphism holds for any choice of an
  ideal $\fa$ in $R$ that reduces to $\fb$ in $S$, i.e.\ $S/\fb \cong
  R/\fa$ as $R$-modules. Since $\fb\subset S$ is non-zero, the ideal
  $\fa$ properly contains $\fp$ and hence it is not in $\mcU$. That
  explains the vanishing of Ext.
\end{proof}

\begin{Claim5}
  One has $\hh(\tau^{\gs d}X)=0$.
\end{Claim5}

\begin{proof}
  By construction one has $\hh^{i}(\tau^{\gs d}X )=0$ for $i\le
  d$. Apply $\RHom_{S}(Q,-)$ to the exact triangle
  \eqref{eq:truncate}. By Claim 3, using that $Q$-vector spaces are
  injective $S$-modules, one has
  \begin{align*}
    \Ext^*_S(Q,\tau^{\gs d}X) &\cong \Ext^*_S(Q,\hh(\tau^{\gs d}X))\\
    &\cong \Hom_S(Q,\hh(\tau^{\gs d}X))\\ &\cong \hh(\tau^{\gs d}X)\:.
  \end{align*}
  For $i>d$, Claim 4 yields $\hh^i(\RHom_S(Q,\tau^{\les d}X)) = 0$,
  and together with the computation above, this explains the first two
  isomorphisms in the next chain
  \begin{align*}
    \hh^{i}(\tau^{\gs d}X)
    & \cong \Ext^{i}_{S}(Q,\tau^{\gs d}X) \\
    &\cong \Ext^{i}_{S}(Q,X) \\
    &\cong \Ext^{i}_{R}(k(\fp),M) \\
    &=0\:.
  \end{align*}
  The third isomorphism follows from Hom-tensor adjunction, recalling
  that $Q = S_{(0)}$ as an $R$-module is $(R/\fp)_{\fp/\fp} \cong
  k(\fp)$. The vanishing of Ext is by hypothesis.
\end{proof}

Finally, from Claim 5 and \eqref{eq:truncate} one gets the second
isomorphism below
\[
\Ext^{i}_{R}(R/\fp,M) \cong \hh^{i}(X) \cong \hh^{i}(\tau^{\les
  d}X)\,;
\]
the first one holds by the definition of $X$. Thus one has
$\Ext^{i}_{R}(R/\fp,M)=0$ for all $i > d$, and this contradicts the
assumption that $\fp$ is in~$\mcU$.

This completes the proof of Theorem~\ref{thm:injdim}. \qed

\medskip

To use Theorem~\ref{thm:injdim} to verify injectivity of an $R$-module $M$ one would have to check vanishing of $\Ext^i_R(k(\fp),M)$, not only for all prime ideals $\fp$ but also for all $i>0$. However, building on this result, in recent work with  Marley~\cite{LWCSITM} we have been able to prove that it suffices to verify the vanishing for a single $i$, \emph{as long as $i$ is large enough}.
The example below illustrates that such a restriction is needed.

\begin{example}
  \label{ex:notone}
  If $R$ is a complete local ring with $\operatorname{depth} R\ge 2$,
  then one has
  \[
  \Ext_R^{1}(k(\fp),R)=0\ \text{ for every prime ideal $\fp$ in
    $R$}\:.
  \]
  Indeed, if $\fp$ is the maximal ideal of $R$, then vanishing holds
  by the assumption $\operatorname{depth} R\ge 2$, and for every
  non-maximal prime $\fp$ one has $\Ext_{R}^{i}(k(\fp),R)=0$ for all
  $i$; see \cite[Example~4.20]{BIK-12} and (\ref{cosupp}).
\end{example}

The next example illustrates that the vanishing of
$\Ext_{R}^{i}(k(\fp),R)$ does not imply that of
$\Ext_{R}^{i}(R/\fp,R)$ and $\Ext_{R_{\fp}}^{i}(k(\fp),R_{\fp})$, and
vice versa. Thus Theorem~\ref{thm:injdim} is not obviously a
consequence of Baer's criterion, nor does it subsume it.

\begin{example}
  \label{ex:QQ}
  Let $R$ be as in Example~\ref{ex:notone} and $\fp$ a prime ideal
  minimal over $(r)$ where $r$ is not a zero divisor.  In this case,
  both $\Ext_{R}^{1}(R/\fp,R)$ and
  $\Ext_{R_{\fp}}^{1}(k(\fp),R_{\fp})$ are nonzero, whilst
  $\Ext^{1}_{R}(k(\fp),R)=0$.

  On the other hand, $\Ext^1_\bbZ(\bbQ,\bbZ)$ is nonzero, whilst
  $\Ext^{1}_{\bbZ}(\bbZ,\bbZ) = 0 = \Ext^1_\bbQ(\bbQ,\bbQ)$.
\end{example}

The analogue of Theorem~\ref{thm:injdim} for flat dimension is
well-known and easier to verify.

\begin{remark}
  Let $M$ be an $R$-complex. For each prime ideal $\fp$ and integer
  $i$ there is a natural isomorphism
  \[
  \Tor_i^R(k(\fp),M)\cong \Tor_i^{R_\fp}(k(\fp),M_\fp) \,.
  \]
  It thus follows from \cite[Proposition 5.3.F]{LLAHBF91} that if
  there exists an integer $d$ such that $\Tor_i^R(k(\fp),M)=0$ for
  $i>d$ and each prime $\fp$, then the flat dimension of $M$ is at
  most $d$. However, Theorem~\ref{thm:injdim} does not follow, it
  seems, from this result by standard injective--flat duality.
\end{remark}

\section{Applications}
\label{se:applications}
We present some applications of Theorem \ref{thm:injdim}. The first
one improves \cite[Theorem~2.2]{LWCFKk} in two directions: There is no
assumption on the projective dimension of flat modules, and an
extension ring is replaced by a module.

\begin{corollary}
  \label{cor:injdim}
  For every $R$-complex $M$ and every faithfully flat $R$-module $F$
  there is an equality
  \begin{equation*}
    \injdim_R \RHom_R(F,M) = \injdim_R M\:.
  \end{equation*}
  In particular, $M$ is acyclic if and only if $\RHom_R(F,M)$ is
  acyclic.
\end{corollary}

\begin{proof}
  For every prime ideal $\fp$ in $R$ and every integer $i$ one has
  \begin{equation*}
    \Ext_{R}^{i}(k(\fp),\RHom_{R}(F,M))
    \cong \Ext_{R}^{i}(F \otimes_R k(\fp),M)
  \end{equation*}
  by adjunction and flatness of $F$. Observe that as an $R$-module
  $F\otimes_{R}k(\fp)$ is a direct sum of copies of $k(\fp)$; it is
  non-zero because $F$ is faithfully flat. It follows that
  $\Ext_{R}^{i}(k(\fp),\RHom_{R}(F,M))$ is zero if and only if
  $\Ext_{R}^{i}(k(\fp),M)$ is zero. The equality of injective
  dimensions now follows from Theorem \ref{thm:injdim}.

  In view of the equality, the statement about acyclicity is trivial
  as $M$ is acyclic if and only if $0$ is a semi-injective resolution
  of $M$ if and only if $\injdim_R M$ is $-\infty$.
\end{proof}

Let $F$ be a flat $R$-module. A module $F\otimes_RM$ is flat if $M$ is
flat, and the converse holds if $F$ is faithfully flat; this is
standard. It is equally standard that the module $\Hom_R(F,M)$ is
injective if $M$ is injective. The next remark provides something
close to a converse; Example~\ref{ex:notone} suggests that the
hypotheses are optimal.

\begin{remark}
  \label{rmk:cobasechange}
  Let $F$ be a faithfully flat $R$-module. If $M$ is an $R$-module
  with $\Ext^{i}_{R}(F,M) =0$ for all $i>0$, then $\RHom_{R}(F,M)$ is
  isomorphic to $\Hom_{R}(F,M)$ in the derived category over
  $R$. Thus, for such a module Corollary \ref{cor:injdim} asserts that
  $\Hom_R(F,M)$ is injective if and only if $M$ is injective. This
  improves the Main Theorem in \cite{LWCFKk}; see also Theorem
  \ref{thm:nninj}.
\end{remark}

The only other result in this direction we are aware of is the Main
Theorem in \cite{LWCFKk}. It deals with the special case where $F$ is
a faithfully flat $R$-algebra, and the proof relies heavily on
\cite[Theorem~4.5]{BIK-12} in the form recovered by
Corollary~\ref{cor:detection}.

This points to our next application, which involves the notion of
cosupport introduced in \cite{BIK-12}, in a form justified by
\cite[Proposition~4.4]{BIK-12}.  The \emph{cosupport} of an
$R$-complex $M$ is the subset of $\Spec R$ given by
\begin{equation}
  \label{cosupp}
  \cosupp_{R}M = \{\fp\in\Spec R\mid \hh(\RHom_R(k(\fp),M)) \ne 0\}\:.
\end{equation}

The next result is \cite[Theorem~4.5]{BIK-12} applied to the derived
category over $R$. The proof of \emph{op.~cit.}\ builds on the
techniques developed in~\cite{BIK-08,BIK-12} to apply to triangulated
categories equipped with ring actions.

\begin{corollary}
  \label{cor:detection}
  An $R$-complex $M$ has $\cosupp_{R}M=\emptyset$ if and only if
  $\hh(M)=0$.
\end{corollary}

\begin{proof}
  The ``if'' is trivial, and the converse holds by
  Theorem~\ref{thm:injdim} when one recalls that $\hh^i(M)\ne 0$
  implies $\injdim_R M \ge i$.
\end{proof}

\begin{remark}
  One can deduce the preceding corollary also from Neeman's
  classification~\cite[Theorem~2.8]{ANm92} of the localizing
  subcategories of the derived category over $R$. Indeed, the
  subcategory of the derived category consisting of $R$-complexes $X$
  with $\Ext^{*}_{R}(X,M)=0$ is localizing. Thus, if it contains
  $k(\fp)$ for each $\fp$ in $\Spec R$, then it must contain $R$, by
  \emph{op.~cit.}, that is to say, $\hh(M)=0$.

  Conversely, Corollary~\ref{cor:detection} can be used to deduce
  Neeman's classification, by mimicking the proof of
  \cite[Theorem~6.1]{BIKP}. The crucial additional observation needed
  to do so is that for $R$-complexes $M$ and $N$, there is an equality
  \[
  \cosupp_{R}\RHom_{R}(M,N) = \supp_{R}M \cap \cosupp_{R}N\:.
  \]
  It follows from two applications of the standard adjunction:
  \begin{align*}
    \hh(\RHom_R(k(\fp),&\RHom_R(M,N))) \\
    &\cong \hh(\RHom_{k(\fp)}(k(\fp)\lotimes_R M ,\RHom_R(k(\fp),N))) \\
    &\cong \Hom_{k(\fp)}(\hh(k(\fp)\lotimes_R M)
    ,\hh(\RHom_R(k(\fp),N)))\:.
  \end{align*}
\end{remark}

\section{Non-noetherian rings}
\label{sec:nonnoetherian}
In this section we establish, over certain not necessarily noetherian
rings, a characterization of injective modules in the vein of
\cite{LWCFKk}; see also Remark~\ref{rmk:cobasechange}. This involves
the following invariant:
\begin{equation*}
  \splf R = \sup\{\projdim_RF \mid F\ \text{\rm is a flat $R$-module}\}\:.
\end{equation*}
A direct sum of flat modules is flat with $\projdim(\bigoplus_{i\in
  I}F_i) = \sup_{i\in I}\{\projdim F_i\}$, so the invariant $\splf R$
is finite if and only if every flat $R$-module has finite projective
dimension. With a nod to Bass' \cite[Theorem P]{HBs60}, a ring with
$\splf R \le d$ is also called a \emph{$d$-perfect} ring. If $R$ has
cardinality at most $\aleph_n$ for some natural number $n$, then one
has $\splf R \le n+1$ by a result of Gruson and Jensen
\cite[Theorem~7.10]{LGrCUJ81}. Osofsky \cite[3.1]{BLO70} has examples
of rings for which the splf invariant is infinite.

\begin{lemma}
  \label{lem:bounded}
  Let $R$ be a commutative ring with $\splf R < \infty$ and let $S$ be
  a faithfully flat $R$-algebra. An $R$-complex $M$ with
  $\hh^{i}(M)=0$ for all $i \gg 0$ is acyclic if and only if
  $\RHom_R(S,M)$ is acyclic.
\end{lemma}

\begin{proof}
  The ``only if'' is trivial, so assume that $\RHom_R(S,M)$ is
  acyclic. As $\hh(M)$ is bounded above, we may assume that
  $\hh^i(M)=0$ holds for all $i>0$, and it suffices to prove that also
  $\hh^0(M)=0$. Set $d := \splf R$.

  Application of $\RHom_R(-,M)$ to the exact sequence \mbox{$0 \to R
    \to S \to S/R \to 0$} yields $M \cong \Sigma\RHom_R(S/R,M)$ in the
  derived category over $R$. Repeated use of this isomorphism and
  adjunction yields $M \cong \Sigma^{d+1}\RHom_R((S/R)^{\otimes
    d+1},M)$. As $S$ is faithfully flat over $R$, the module $S/R$ is
  flat, and hence so are its tensor powers. Thus, the module
  $(S/R)^{\otimes d+1}$ has projective dimension at most $d$ and,
  therefore, $\hh^i(\RHom_R((S/R)^{\otimes d+1},M))=0$ holds for all
  $i>d$. In particular,
  \begin{align*}
    \hh^0(M) &\cong \hh^0(\Sigma^{d+1}\RHom_R((S/R)^{\otimes d+1},M))\\
    & = \hh^{d+1}(\RHom_R((S/R)^{\otimes d+1},M))\\
    &= 0\:.\qedhere
  \end{align*}
\end{proof}

\begin{proposition}
  \label{prp:trivial}
  Let $R$ be a commutative ring with $\splf R < \infty$ and let $S$ be
  a faithfully flat $R$-algebra of projective dimension at most
  $1$. An $R$-complex $M$ is acyclic if and only if $\RHom_R(S,M)$ is
  acyclic.
\end{proposition}

\begin{proof}
  The ``only if'' is trivial, so assume that $\RHom_R(S,M)$ is
  acyclic. To prove that $M$ is acyclic, we show that
  $\hh^0(\Sigma^nM)=0$ holds for all $n\in\bbZ$. Fix $n$ and let
  $\Sigma^nM \to I$ be a semi-injective resolution; the assumption is
  now $\hh(\Hom_R(S,I))=0$ and the goal is to prove $\hh^0(I)=0$.

  The soft truncation
  \begin{equation*}
    \tau^{\le 1}\Hom_R(S,I) = \cdots \to \Hom_R(S,I)^{-1}\to \Hom_R(S,I)^{0} \to \zz^1(\Hom_R(S,I)) \to 0
  \end{equation*}
  is acyclic, and by left-exactness of Hom one has $\tau^{\le
    1}\Hom_R(S,I) = \Hom_R(S,\tau^{\le 1}I)$. Further, still by
  acyclicity of $\Hom_R(S,I)$, there is an equality
  \[
  \bb^2(\Hom_R(S,I)) = \Hom_R(S,\bb^2(I))\,.
  \]
  Thus, the functor $\Hom_R(S,-)$ leaves the sequence $0 \to \zz^1(I)
  \to I^1 \to \bb^2(I) \to 0$ exact, and that implies vanishing of
  $\Ext_R^1(S,\zz^1(I))$.

  Let $\pi\colon P \to S$ be a projective resolution over $R$ with
  $P_i=0$ for $i >1$. Consider its mapping cone
  \[
  A = 0 \lra P_1 \lra P_0 \lra S \lra 0\,.
  \]
  As $\Hom_R(A,I^n)$ is exact for every $n$ and $\Hom_R(A,\zz^1(I))$
  is exact by vanishing of $\Ext_R^1(S,\zz^1(I))$, it follows from
  \cite[Lemma (2.5)]{CFH-06} that $\Hom_R(A,\tau^{\le 1}I)$ is
  acyclic. Thus, $\Hom_R(\pi,\tau^{\le 1}I)$ yields an isomorphism
  $\RHom_R(S,\tau^{\le 1}I) \cong \Hom_R(S,\tau^{\le 1}I)$ in the
  derived category, and the latter complex is acyclic. Now
  Lemma~\ref{lem:bounded} yields $\hh(\tau^{\le 1}I)=0$, in particular
  $\hh^0(I) = \hh^0(\tau^{\le 1}I) = 0$.
\end{proof}

\begin{theorem}
  \label{thm:nninj}
  Let $R$ be a commutative ring with $\splf R < \infty$, let $S$ be a
  faithfully flat $R$-algebra of projective dimension at most $1$, and
  let $M$ be an $R$-module. If $\Ext^1_R(S,M)=0$ and the $S$-module
  $\Hom_R(S,M)$ is injective, then $M$ is injective.
\end{theorem}

\begin{proof}
  The proof of \cite[Theorem~1.7]{LWCFKk} applies with one
  modification: in place of \cite[1.5]{LWCFKk}---at heart a reference
  to \cite[Theorem~4.5]{BIK-12}---one invokes Proposition
  \ref{prp:trivial}.
\end{proof}

\begin{remark}
  The assumption in Theorem~\ref{thm:nninj} that the flat $R$-algebra
  $S$ has projective dimension at most $1$ is satisfied if
  \begin{itemize}
  \item $R$ is countable; see \cite[Theorem~7.10]{LGrCUJ81}.
  \item $S$ is countably related; in particular, if every ideal in $R$
    is countably generated, and $S$ is countably generated as an
    $R$-module; see Osofsky~\cite[Lemma 1.2]{BLO68} and
    Jensen~\cite[Lemma 2]{CUJ66}.
  \end{itemize}
\end{remark}

\section*{Acknowledgments}

L.W.C.\ thanks Fatih K\" oksal for conversations related to this work;
the joint paper \cite{LWCFKk} provided much inspiration. S.B.I.\
thanks Dave Benson, Henning Krause, and Julia Pevtsova for discussions
related to this work. The statement of Theorem~\ref{thm:injdim}
emerged out of an on-going collaboration with them. We also thank Tom
Marley for pointing out an error in an earlier version of Remark 2.3.

\providecommand{\arxiv}[2][AC]{\mbox{\href{http://arxiv.org/abs/#2}{\sf
      arXiv:#2 [math.#1]}}}
\providecommand{\MR}[1]{\mbox{\href{http://www.ams.org/mathscinet-getitem?mr=#1}{#1}}}
\renewcommand{\MR}[1]{\mbox{\href{http://www.ams.org/mathscinet-getitem?mr=#1}{#1}}}
\providecommand{\MR}{\relax\ifhmode\unskip\space\fi MR }
\providecommand{\MRhref}[2]{%
  \href{http://www.ams.org/mathscinet-getitem?mr=#1}{#2} }
\providecommand{\href}[2]{#2}

\end{document}